\documentclass[12pt]{amsart}
\usepackage[utf8]{inputenc}
\usepackage[ngerman,english]{babel}

\title{Combinatorics of Correlated Equilibria}

\author{Marie-Charlotte Brandenburg}
\address{}
\email{} 

\author{Benjamin Hollering}
\address{}
\email{} 

\author{Irem Portakal}
\address{}
\email{} 

\date{}
\usepackage[dvipsnames]{xcolor}
\usepackage{fullpage}

\usepackage{multirow,array}
\usepackage{graphicx}
\usepackage{amsmath}             
\usepackage{amsfonts} 
\usepackage{tikz-cd} 
\usepackage{soul}
\usepackage{multirow}
\usepackage[nomain,style=long3col,sort=use,section=section]{glossaries}

\usepackage{setspace}

\usepackage{amsthm}               
\usepackage{amssymb}              
\usepackage{tikz}                       
\usepackage[all]{xy}
\usepackage{verbatim}
\usetikzlibrary{plotmarks}
\usetikzlibrary{shapes}
\usetikzlibrary{arrows}

\usepackage{graphicx}
\usepackage{caption}
\usepackage{subcaption}

\usepackage{upgreek}
\usepackage{textcomp, gensymb}
\usepackage{capt-of}
\usepackage{ stmaryrd }
\usepackage{colortbl}
\usepackage{multirow}
\usepackage{hhline}

\usepackage[pdftex]{hyperref}
\hypersetup{
    colorlinks=true,
    linkcolor={MidnightBlue},
    citecolor={magenta},
    urlcolor={MidnightBlue}
}

\usepackage{caption}
\usepackage[nameinlink,capitalise]{cleveref}
\usepackage{float}

\usepackage{todonotes}

\numberwithin{equation}{section}
\theoremstyle{plain}

\newtheorem{theorem}{Theorem}[section]
\newtheorem{conjecture}{Conjecture}[section]

\newtheorem{proposition}[theorem]{Proposition}

\newtheorem{manualtheoreminner}{}
\newenvironment{manualtheorem}[1]{%
  \renewcommand\themanualtheoreminner{#1}%
  \manualtheoreminner
}{\endmanualtheoreminner}

\theoremstyle{definition}

\newtheorem{remark}[theorem]{Remark}
\newtheorem{example}[theorem]{Example}

\newtheorem{definition}[theorem]{Definition}

\newtheorem{construction}[theorem]{Construction}

\newcommand{\R}{\mathbb R}
\newcommand{\N}{\mathbb N}

\newcommand{\ma}{\begin{pmatrix}}
\newcommand{\trix}{\end{pmatrix}}
\newcommand{\sma}{\left(\begin{smallmatrix}}
\newcommand{\strix}{\end{smallmatrix}\right)}

\newcommand{\CES}{\mathcal{S}}

\allowdisplaybreaks

\Crefname{page}{page}{page}

\begin{document}
\begin{abstract}
We study the correlated equilibrium polytope $P_G$ of a game $G$ from a combinatorial point of view. We introduce the region of full-dimensionality for this class of polytopes and prove that it is a semialgebraic set for any game. Using a stratification via oriented matroids, we propose a structured method for describing the possible combinatorial types of $P_G$, and show that for $(2 \times n)$-games, the algebraic boundary of the stratification is a union of coordinate hyperplanes and binomial hypersurfaces. Finally, we provide a computational proof that there exists a unique combinatorial type of maximal dimension for generic $(2 \times 3)$-games.
\end{abstract}

\maketitle

\section{Introduction}

In 1950, Nash published a very influential two-page article \cite{nash50} proving the existence of a Nash equilibrium for any finite game.  This opened many new fronts, not only in game theory, but also in areas such as economics, computer science, evolutionary biology, quantum mechanics, and the social sciences. To study Nash equilibria one assumes that the actions of the players are independent and completely separated from any exterior influence. Moreover, these can be described as a system of multilinear equations \cite[Section 6]{Sturmfels02solvingsystems}. However, there exist cases where a Nash equilibrium fails to predict the most beneficial outcome (e.g.\ Pareto optimality) for all players. There are several approaches, rooted in the concept of a {\em dependency equilibrium,} which generalize Nash equilibria by imposing dependencies between the actions of players.
This class of equilibria has been studied from the point of view of algebraic statistics and computational algebraic geometry \cite{spohn2003,PS22,PSA24,PSA23}. 
On the other hand, Aumann introduced the concept of a {\em correlated equilibrium}, which assumes that there is an external correlation device such as a mediator or some other physical source. The resulting correlated equilibria are probability distributions of recommended joint strategies \cite{aumann1,aumann2}. In contrast to Nash equilibria and dependency equilibria, correlated equilibria are significantly less computationally expensive, since they only require solving a linear program \cite{Pap05}. 
In other words, the set of such equilibria can be described by linear inequalities in the probability simplex and thus form a convex polytope called the \emph{correlated equilibrium polytope}. In this article, we study combinatorial properties of correlated equilibrium polytopes with methods from convex geometry and real algebraic geometry. 

\begin{figure}[t]
 \begin{minipage}{6in}
  \centering
  \raisebox{-0.5\height}{\includegraphics[height=1.3in]{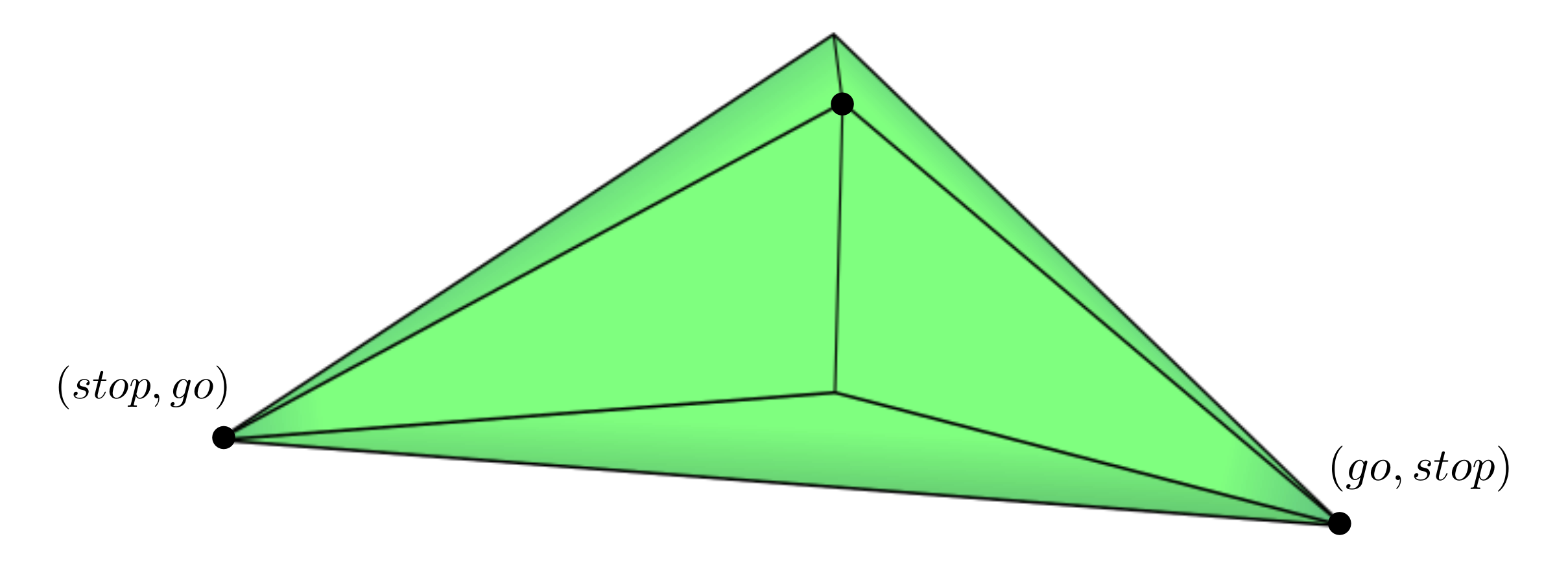}}
  \hspace*{.1in}
  \raisebox{-0.5\height}{\includegraphics[height=1.8in]{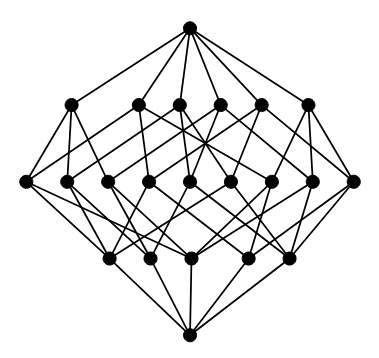}}
\end{minipage}
    \caption{The correlated equilibrium polytope of the Traffic Lights example is a bipyramid where three of its vertices are Nash equilibria. Its $f$-vector is
    $(1,5,9,6,1)$ and its face lattice can be seen on the right.}
    \label{fig:traffic-lights}
 \end{figure}

We illustrate the concept of correlated equilibrium in an example: Two cars meet at a crossing. Both drivers would like to continue to drive, but, even more importantly, would also like to avoid a car crash. Thus, each of the drivers prefers not to drive in case the other chooses to drive. We make the assumptions that both drivers are unable to communicate with each other.
This is a classic game in game theory known as Chicken game or Hawk-Dove game, and we formalize this in \Cref{ex:traffic-lights-1,ex:traffic-lights-2,ex:traffic-lights-3,ex: correlated polytope}. However, this situation changes drastically if there is a traffic light installed at this crossing. We can view the traffic light as a neutral exterior party, that  gives a recommendation to each player in showing green or red lights; here we assume that each driver only knows their own given recommendation.
If a fixed driver is given such a recommendation (for example a red light), the driver now ponders about deviating from this recommendation in benefit of their own (selfish) good, assuming that the other player adheres to their own given recommendation. If both drivers decide not to deviate from the recommendation given by the traffic lights, a correlated equilibrium is achieved. \\

To our knowledge, there are no articles concerning the combinatorics of correlated equilibrium polytopes in the language of convex geometry up to this date, despite the fact that the concept of correlated equilibria is a topic of extensive research in economics and game theory. In this article, we study this class of polytopes from a combinatorial perspective. 
In general, the correlated equilibrium polytope can exhibit a great variety of distinct combinatorial structures. This is not surprising as it is proven that any convex polytope can be realized as the correlated equilibrium payoffs of an $n$-player game \cite{LESEV11}. Even classifying necessary conditions under which the correlated equilibrium polytope is of maximal dimension is highly nontrivial \cite{viossat_elementarygamesgames}. This is also an interesting question from a game theoretical perspective, in particular for the computation of Nash equilibria which is PPAD-complete \cite{daskalakis2009complexity}. In this case, Nash equilibria lie on the relative boundary of the correlated equilibrium polytope (\Cref{prop: nash equilibria on boundary}). However, there are examples of games where the correlated equilibrium polytope is not full dimensional and Nash equilibria lie in the relative interior \cite[Proposition 2]{nau_geometrynashequilibria}. To study full-dimensional correlated equilibrium polytopes, we introduce the \emph{region of full-dimensionality}, a set that classifies under which conditions the correlated equilibrium polytope has maximal dimension.

\begingroup
\hypersetup{hidelinks}
    \begin{manualtheorem}{\Cref{th:full-dim-semialgebraic}}
        The region of full-dimensionality is a semialgebraic set and can be explicitly described.
    \end{manualtheorem}
\endgroup

The full description of this semialgebraic set can be found on \cpageref{th:full-dim-semialgebraic}.
We continue the study of the combinatorial structure of correlated equilibrium polytopes by introducing a linear space called the \emph{correlated equilibrium space}, and consider an oriented matroid stratification inside this space. This is a stratification of the linear space, in which regions give rise to the different combinatorial types of correlated equilibrium polytopes. 
We study the algebraic boundary of the stratification for $(2\times n)$-games, which turns out to be generated by binomials corresponding to $(2 \times 2)$-minors and $(4 \times 4)$-minors of a certain matrix (\Cref{th:binomials}). 
These investigations yield novel insights into the possible combinatorial types of generic $(2 \times 3)$-games. The term \emph{generic games} is used in relation to the algebraic boundary explained in \Cref{th:binomials}.
\begingroup
  \hypersetup{hidelinks}
    \begin{manualtheorem}{\Cref{thm:2x3typesFromStrata}}
    Let $G$ be a generic $(2 \times 3)$-game and $P_G$ be the associated correlated equilibrium polytope. Then one of the following holds:
    \begin{itemize}
        \item $P_G$ is a point,
        \item $P_G$ is of maximal dimension $5$ and of a unique combinatorial type,
        \item There exists a $(2 \times 2)$-game $G'$ such that $P_{G'}$ has maximal dimension $3$ and is combinatorially equivalent to $P_G$. 
    \end{itemize}
    \end{manualtheorem}
\endgroup

Supported by the above theorem and our computations \cite{mathrepo} for generic $(2\times n)$-games $G$ (where $n\leq5$), we conjecture that if $P_G$ is not of maximal dimension, then there exists a smaller $(2 \times k)$-game $G'$ with $k <n$ such that $P_{G'}$ is a full-dimensional polytope and $P_G$ is combinatorially equivalent to $P_{G'}$ (\Cref{conj:2xntypes}). 

\subsection*{Overview}
We first provide a short introduction for the necessary concepts from game theory in \Cref{sec:correlated-polytope}, including correlated equilibria and Nash equilibria. In \Cref{sec:correlated-cone} we describe the correlated equilibrium cone, a convex polyhedral cone which captures the geometry of the correlated equilibrium polytope, and describe the correlated equilibrium space. 
We study the region of full-dimensionality in \Cref{sec:full-dim}. Finally, we consider the possible combinatorial types through the oriented matroid stratification in \Cref{sec:comb-types}. All results are illustrated with examples for games of types $(2 \times 2), \ (2 \times 3)$ and, whenever possible, games of type $(2 \times 2 \times 2)$.
All referenced code and a detailed study of some examples with visualizations can be found on MathRepo \cite{mathrepo}:

\begin{center}
    \url{https://mathrepo.mis.mpg.de/correlated-equilibrium}
\end{center}

\subsection*{Acknowledgements}
The authors would like to thank Rainer Sinn and Bernd Sturmfels for several helpful discussions. 

\section{The Correlated Equilibrium Polytope}\label{sec:correlated-polytope}

Let $n$ be the number of players in a normal-form game. Each player $i \in [n]$ has a fixed set of \emph{pure strategies} $s^{(i)}_{1},\dots s^{(i)}_{d_i}, d_i \in \N$. 
It is practical to think of each strategy as a single move that a player can play, and all players perform their single move simultaneously. Afterwards, the game is over, so the choices of the possible moves can be seen as the outcome of the game.
A \emph{pure joint strategy} is a tuple $s_{j_1 \dots j_n} = (s^{(1)}_{j_1}, \dots, s^{(n)}_{j_n})$ of strategies, where each player $i \in [n]$ chooses to play a fixed strategy $s^{(i)}_{j_i}$ with $j_i \in [d_i], i \in [n]$. The \emph{payoff} $X^{(i)}_{j_1 \dots j_n} \in \R$ of player $i$ at $s_{j_1 \dots j_n}$ is the quantity of how beneficial player $i$ values the combination $(s^{(1)}_{j_1},\dots, s^{(n)}_{j_n})$ of strategies as outcome of the game. 
A \emph{mixed strategy} of a single player $i$ is an action with probability $p^{(i)} = (p^{(i)}_{j_1}, \dots, p^{(i)}_{j_{d_i}})$, i.e.\ $p^{(i)}_{j_1}, \dots, p^{(i)}_{j_{d_i}} \geq 0$ and $\sum_{k = 1}^{d_i} p^{(i)}_{j_k} = 1$. 
We can view this geometrically as a point in the $(d_i - 1)$-dimensional probability simplex $\Delta_{d_i -1}$.

Formally, a $(d_1 \times \dots \times d_n)$-game in normal form is a tuple $G = (n, S, X)$, where $S =~(S^{(1)}, \dots , S^{(n)})$ is the collection of strategies $S^{(i)} = (s^{(i)}_1 , \dots , s^{(i)}_{d_i})$ of all players, and $X =~(X^{(1)}, \dots, X^{(n)})$ is the collection of all $(d_1 \times \cdots \times d_n)$-payoff tensors.
In particular, if $n=2$, then each $X^{(i)}$ is a $(d_1 \times d_2)$-matrix and called the \emph{payoff matrix} of player $i$. 

\begin{example}[Traffic Lights]\label{ex:traffic-lights-1}
Recall the example from the introduction, in which two cars meet at a crossing and would like to avoid a car crash. Formally, each player $i \in [2]$ has the strategies $s^{(i)}_1 = $ ``go'' and $s^{(i)}_2 = $ ``stop''. The bimatrix below shows the payoffs of both players simultaneously, where each entry is a tuple $(X^{(1)}_{j_1 j_2}, X^{(2)}_{j_1 j_2})$ of the payoff of each player for the tuple of strategies $(s^{(1)}_{j_1}, s^{(2)}_{j_2}), j_1, j_2 \in [2]$ as the expected outcome of the game.

 \begin{table}[H]
	\setlength{\extrarowheight}{2pt}
	\begin{tabular}{cc|c|c|}
		& \multicolumn{1}{c}{} & \multicolumn{2}{c}{Player $2$}\\
		& \multicolumn{1}{c}{} & \multicolumn{1}{c}{go}  & \multicolumn{1}{c}{stop} \\\cline{3-4}
		\multirow{2}*{Player $1$}  & go & $(-99, -99)$ & $(1,0)$ \\\cline{3-4}
		& stop & $(0,1)$ & $ (0,0) $ \\\cline{3-4} \\
	\end{tabular} \hspace*{2.8cm}
\end{table}

We may interpret the given payoffs such that each of the players prefers to go (with payoff $X^{(1)}_{12} = X^{(2)}_{21}=1$), if the other driver stops. However, both drivers choosing to drive is the worst possible outcome for each of the players individually (with payoff $X^{(i)}_{11}=-99$ for both players). 
Since there is no interaction between the players, and the payoffs of a car crash are very negative, the players are very likely not to risk a move, although this is not optimal for either of these players ($X^{(i)}_{22}=0$ for both players). This is the motivation for introducing the concept of correlated equilibrium, in which the assumptions allow a dependence of the moves of the players by the suggestion of an external party e.g.\ a traffic light.
We continue with this in \Cref{ex:traffic-lights-2,ex:traffic-lights-3}.
\end{example}

We now allow a third, independent party to influence the game from the outside. 
Let $\tilde{S} = \{s_{j_1 \dots j_n} \mid j_i \in [d_i], i \in [n]\}$ be the set of all pure joint strategies of the game. The external party draws such a pure joint strategy with probability $p_{j_1 \dots j_n} \geq 0$, called a \emph{mixed joint strategy}. Such a joint probability distribution is a vector (or a tensor) $p = (p_{j_1 \dots j_n} \mid j_i \in [d_i], i \in [n])$, such that $\sum_{j_1=1}^{d_1} \cdots \sum_{j_n=1}^{d_n} p_{j_1 \dots j_n} = 1$. The set of all joint probability distributions is the probability simplex $\Delta_{d_1 \cdots d_n -1}$.

\begin{definition}[Correlated Equilibrium]
	Let $G$ be a game with payoffs $X = (X^{(1)},\dots,X^{(n)})$. A point $p \in \Delta_{d_1 \cdots d_n -1}$ is a \emph{correlated equilibrium} if and only if 
	\begin{equation}\label{eq: correlated equilibria}
		\sum_{j_1=1}^{d_1} \cdots \widehat{\sum_{j_{i}=1}^{d_{i}}} \cdots \sum_{j_{n}=1}^{d_{n}}    \left(  X^{(i)}_{j_1\cdots j_{i-1} k j_{i+1} \cdots j_n}  - X^{(i)}_{j_1 \cdots j_{i-1} l j_{i+1} \cdots, j_n} \right) p_{j_1  \cdots j_{i-1} k j_{i+1} \cdots j_n} \geq 0.
	\end{equation}
	for all $k,l \in [d_i]$, and for all $i \in [n]$.
	The linear inequalities in \eqref{eq: correlated equilibria} together with the linear constrains 
	\[
        p_{j_1 \dots j_n} \geq 0 \text{ for } j_i \in [d_i], i \in [n], \text{ and } \sum_{j_1=1}^{d_1} \cdots \sum_{j_n=1}^{d_n} p_{j_1 \dots j_n} = 1\] 
define the set of all correlated equilibria of the game. The set of all such equilibria is the \emph{correlated equilibrium polytope} $P_G$ of the game $G$.
\end{definition}

The ambient space of the polytope $P_G$ has dimension $d_1 \cdots d_n$. By definition, the maximal dimension that $P_G$ can achieve is $d_1 \cdots d_n -1$. In the literature, in this case, $P_G$ is often called \emph{full-dimensional}. To avoid confusion with conventions in convex geometry, we refer to $P_G$ as having \emph{maximal dimension}.

\begin{example}[Traffic Lights]\label{ex:traffic-lights-2}
We continue with \Cref{ex:traffic-lights-1}. Here, the third party is given by a traffic light at the crossing. The traffic light gives the recommendation ``go'' to a driver if it turns green, and ``stop'' if it turns red. The traffic light draws randomly from one of the combinations of strategies. Let $p_{j_1 j_2}$ be the probability with which the traffic light draws the pure joint strategy $s_{j_1 j_2}$. The point $p = (p_{11},p_{12},p_{21},p_{22})$ is a correlated equilibrium if and only if $p_{j_1 j_2} \geq 0, \ p_{11}+p_{12}+p_{21}+p_{22} =1$ and
\begin{align*}
(X^{(1)}_{\text{22}} - X^{(1)}_{\text{12}})p _{\text{22}} +  (X^{(1)}_{\text{21}} - X^{(1)}_{\text{11}})p _{\text{21}} \geq 0, \qquad (X^{(1)}_{\text{12}} - X^{(1)}_{\text{22}})p _{\text{12}} +  (X^{(1)}_{\text{11}} - X^{(1)}_{\text{21}})p _{\text{11}} \geq 0, \\
(X^{(2)}_{\text{22}} - X^{(2)}_{\text{21}})p _{\text{22}} +  (X^{(2)}_{\text{12}} - X^{(2)}_{\text{11}})p _{\text{12}} \geq 0, \qquad
(X^{(2)}_{\text{21}} - X^{(2)}_{\text{22}})p _{\text{21}} +  (X^{(2)}_{\text{11}} - X^{(2)}_{\text{12}})p _{\text{11}} \geq 0.
\end{align*}
 
With the payoffs as given in the bimatrix in \Cref{ex:traffic-lights-1} these inequalities evaluate to
\begin{align*}
    &- p _{\text{22}} +   99 p _{\text{21}} \geq 0, 
    &p _{\text{12}}   -99 p _{\text{11}} \geq 0, \\
    &- p _{\text{22}} +   99 p _{\text{12}} \geq 0
    &p _{\text{21}} -99 p _{\text{11}} \geq 0.
\end{align*}

The correlated equilibrium polytope, i.e.\ the points $p$ that satisfy these inequalities, has $5$ vertices with coordinates $(0, 0, 1, 0),$
$(0, 1, 0, 0), $
$(\tfrac{1}{10000},\tfrac{99}{10000}, \tfrac{99}{10000}, \tfrac{9801}{10000}),$
 $(0, \tfrac{1}{101}, \tfrac{1}{101}, \tfrac{99}{101}),$
 $(\tfrac{1}{199}, \tfrac{99}{199}, \tfrac{99}{199},0)$. This polytope is depicted in \Cref{fig:traffic-lights-with-labels}.
\end{example}

\begin{figure}[t]
\includegraphics[scale=0.16]{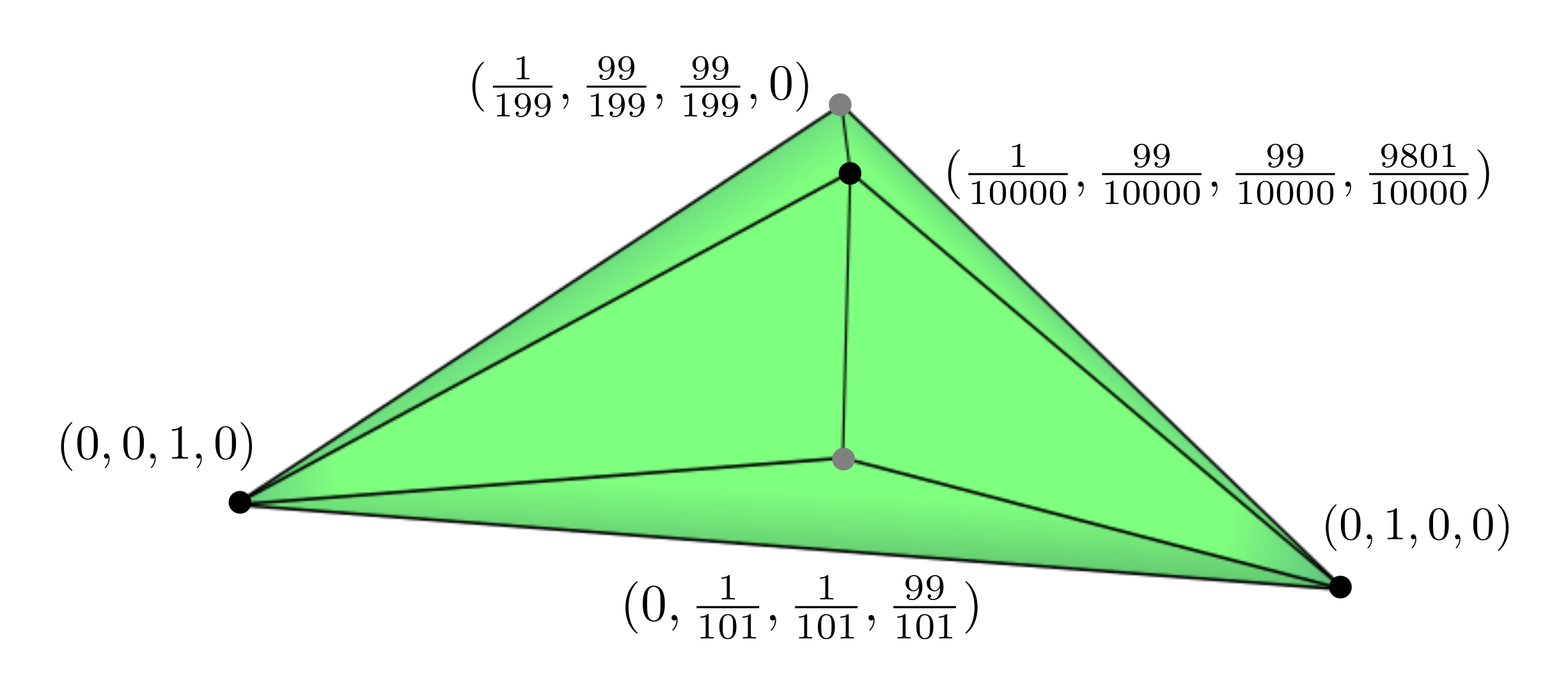}
\caption{The vertices of the correlated equilibrium polytope for the Traffic Lights example (\Cref{ex:traffic-lights-2})}
\label{fig:traffic-lights-with-labels}
\end{figure}

The next proposition shows that two affinely dependent games define the same correlated equilibrium polytope.

\begin{proposition}\label{prop:affine-transformation}
	Any affine linear transformation of the payoff tensors $X^{(i)}$ with positive scalars leaves the polytope $P_G$ invariant. More precisely, let $G = (n, S, X)$ be a game. For each $i \in [n]$, fix $t_i \in \R, \ \lambda_i \in \R_{>0}$ and let $\tilde{X}^{(i)}_{j_1 \dots j_n} = \lambda_i X^{(i)}_{j_1 \dots j_n} + t_i$ for all $j_k \in [d_k], k \in [n]$. Then for the game $\Tilde{G} = (n, \Tilde{X}, S)$ with $\tilde{X} = (\tilde{X}^{(1)}, \dots, \tilde{X}^{(n)})$ it holds that $P_G = P_{\Tilde{G}}$.
\end{proposition}

\begin{proof}
	For each player $i \in [n]$ let $X^{(i)}$ be their payoff tensor, fix $t_i \in \R, \ \lambda_i \in \R_{>0}$, and consider the affine transformation $\tilde{X}^{(i)}_{j_1 \dots j_n} = \lambda_i X^{(i)}_{j_1\dots j_n} + t_i$. Then
	\begin{align*}
		&\sum_{j_1=1}^{d_1} \dots \widehat{\sum_{j_{i}=1}^{d_{i}}}  \dots \sum_{j_{n}=1}^{d_{n}}    \left(  \Tilde{X}^{(i)}_{j_1 \dots  j_{i-1} k j_{i+1}   \dots j_n}  - \Tilde{X}^{(i)}_{j_1 \ldots   j_{i-1} l j_{i+1}     \dots  j_n} \right) p_{j_1  \dots   j_{i-1} k j_{i+1}   \dots j_n} \\
		= &\sum_{j_1=1}^{d_1} \dots \widehat{\sum_{j_{i}=1}^{d_{i}}}  \dots \sum_{j_{n}=1}^{d_{n}}    \left(  \lambda_i X^{(i)}_{j_1 \dots j_{i-1} k j_{i+1}  \dots  j_n} + t_i  - \lambda_i X^{(i)}_{j_1 \dots  j_{i-1} l j_{i+1}   \dots j_n} - t_i \right) p_{j_1  \dots j_{i-1} k j_{i+1} \dots j_n} \\
		= & \lambda_i \sum_{j_1=1}^{d_1} \dots \widehat{\sum_{j_{i}=1}^{d_{i}} }  \dots \sum_{j_{n}=1}^{d_{n}}    \left(  X^{(i)}_{j_1 \dots j_{i-1} k j_{i+1}  \dots  j_n}  - X^{(i)}_{j_1 \dots j_{i-1} l j_{i+1}  \dots j_n} \right) p_{j_1  \dots j_{i-1} k j_{i+1} \dots j_n} \geq 0.
	\end{align*}
	Since $\lambda_i > 0$, this is equivalent to \eqref{eq: correlated equilibria} being nonnegative.
\end{proof}

\begin{figure}[t]   
\includegraphics[width=7cm]{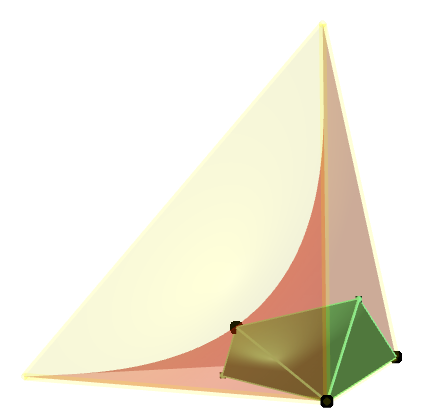}
\caption{A $3$-dimensional correlated equilibrium polytope (green) inside the probability simplex $\Delta_3$ (yellow) for a $(2\times 2)$-game. Its Nash equilibria (black) are the intersection with the Segre variety (red). This picture applies to the Traffic Lights example (\Cref{ex:traffic-lights-1,ex:traffic-lights-2,ex:traffic-lights-3}) as well as the Hawk-Dove game (\Cref{ex: correlated polytope}).}
\label{fig: correlated equilibria polytope} 
\end{figure}

\begin{definition}[Nash Equilibrium]\label{def: Nash equilibrium}
Let $G$ be a game.
A \emph{Nash equilibrium} of $G$ is a point $p \in P_G$ that is a tensor of rank one. More specifically, the set of Nash equilibria are those points in $P_G$ that are also contained in the image of the product map
\begin{align*}
	\varphi: \Delta_{d_1 -1} \times \cdots \times \Delta_{d_n -1} & \rightarrow \Delta_{d_1 \cdots d_n -1} \\
	(p^{(1)}, \dots, p^{(n)}) &\mapsto p^{(1)}_{j_1}\cdot \ldots \cdot p^{(n)}_{j_n}
\end{align*}

The image of this map is the \emph{Segre variety} 
inside $\Delta_{d_1 \cdots d_n -1}$. 
\end{definition}

\Cref{fig: correlated equilibria polytope} shows a $3$-dimensional correlated equilibrium polytope of a $(2 \times 2)$-game together with the Segre variety inside the simplex $\Delta_3$.
We illustrate this in more detail in the following example. 

\begin{example}[Traffic Lights]\label{ex:traffic-lights-3}
We continue with the Traffic Lights example from \Cref{ex:traffic-lights-1,ex:traffic-lights-2}. The Nash equilibria of this game occur as vertices of the correlated equilibrium polytope $P_G$. 
More precisely, they occur as the images of the points $(p^{(1)} , p^{(2)}) \in \Delta_2 \times \Delta_2$ with coordinates $((1, 0), (0,1)), \ ((0, 1), (1,0)) $ and $((\tfrac{1}{100}, \tfrac{99}{100}), ((\tfrac{1}{100}, \tfrac{99}{100}))$ under the product map $\varphi$ from \Cref{def: Nash equilibrium}, which correspond to the three black vertices in \Cref{fig:traffic-lights-with-labels}. With indexing $p = (p_{11},p_{12},p_{21},p_{22})$ the images of the first two points are $p = (0,1,0,0)$ and $(0,0,1,0)$. These are the probability distributions which correspond to the pure joint strategies in which one of the players drives, while the other one stops. The point $p = (\tfrac{1}{10000},\tfrac{99}{10000}, \tfrac{99}{10000},\tfrac{9801}{10000})$ is a probability distribution in which it is most likely that both players stop. 

\end{example}
By \Cref{def: Nash equilibrium}, the set of Nash equilibria is a subset of the correlated equilibrium polytope. However, characterizing Nash equilibria is computationally difficult. Therefore, it is of interest to understand where the Nash equilibria lie relative to the correlated equilibrium polytope \cite{papadimitriou05_computingequilibriamulti}. A game $G$ is called non-trivial if $X^{(i)}_{j_1 \ldots j_{i-1} k j_{i+1}, \dots j_n} \neq X^{(i)}_{j_1 \ldots j_{i-1} l j_{i+1}, \ldots j_n}$ for some player $i \in [n]$ and $k,l \in [d_i]$ with $k \neq l$. The next result states that if $P_G$ is of maximal dimension, then any Nash equilibrium lies on a proper face of $P_G$. 
\begin{proposition}
	[{\cite[Proposition 1]{nau_geometrynashequilibria}}]
	\label{prop: nash equilibria on boundary}
	Let $G$ be a non-trivial game. Then the Nash equilibria lie on a face of the correlated equilibrium polytope $P_G$ of dimension at most $d_1 \cdots  d_n -2$. In particular, if $P_G$ has maximal dimension $d_1 \cdots  d_n -1$, then the Nash equilibria lie on the relative boundary of $P_G$.
\end{proposition}

In order to locate the possible positions of Nash equilibria, it is thus helpful to understand the conditions under which $P_G$ is of maximal dimension. In \Cref{sec:full-dim} we study the region of full-dimensionality, which formalizes these conditions.

\section{The Correlated Equilibrium Cone}\label{sec:correlated-cone}

The combinatorics of the correlated equilibrium polytope is completely determined by the cone given by the inequality constraints \eqref{eq: correlated equilibria}, intersected with the nonnegative orthant. The \emph{correlated equilibrium cone} $C_G \subseteq \R^{d_1 \cdots d_n}$ is the polyhedral cone defined by inequalities $p_{j_1 \dots j_n} \geq 0$ and 
	\begin{equation}\label{eq: correlated equilibria cone}
	\sum_{j_1=1}^{d_1} \cdots \widehat{\sum_{j_{i}=1}^{d_{i}}}  \cdots \sum_{j_{n}=1}^{d_{n}}    \left(  X^{(i)}_{j_1\cdots j_{i-1} k j_{i+1} \cdots j_n}  - X^{(i)}_{j_1 \cdots  j_{i-1} l j_{i+1} \cdots j_n} \right) p_{j_1  \cdots j_{i-1} k j_{i+1} \cdots j_n} \geq 0.
\end{equation}
for all $k,l \in [d_i]$, and for all $i \in [n]$. For each player $i \in [n]$ this defines $d_i(d_i - 1)$ nontrivial inequalities of type \eqref{eq: correlated equilibria cone}. The cone $C_G$ is a convex pointed polyhedral cone. The correlated equilibrium polytope $P_G$ is the intersection of the cone with the hyperplane where the sum of all coordinates equals $1$. Therefore, a facet of $P_G$ is in bijection with a facet of $C_G$, and a vertex of $P_G$ is in bijection with an extremal ray of $C_G$.

We make a substitution of the coefficients in \eqref{eq: correlated equilibria cone}. For each $i \in [n], j_1 \in [d_1], \dots, j_{i-1} \in~[d_{i-1}]$, $k,l \in [d_i], \ j_{i+1} \in [d_{i+1}],\dots , j_n \in [d_n]$ we define
\[
    Y^{(i)}_{j_1\cdots\hat{j_i}\cdots j_n}(k,l) = X^{(i)}_{j_1\cdots j_{i-1} k j_{i+1} \cdots j_n}  - X^{(i)}_{j_1 \cdots j_{i-1} l j_{i+1} \cdots j_n} .
\]
Note that $Y^{(i)}_{j_1 \dots \hat{j_i} \dots j_n}(k,l) = - Y^{(i)}_{j_1 \dots \hat{j_i} \dots j_n}(l,k)$.
Thus, for each player $i \in [n]$ this defines $\binom{d_{i}}{2} \prod_{\substack{k \in [n] \\ k \neq i}} d_k $ distinct variables, so in total this defines 
\[
    M = \sum_{i = 1}^n \binom{d_{i}}{2} \prod_{\substack{k \in [n] \\ k \neq i}} d_k
\]
many variables. Under this substitution, the above inequality becomes
	\begin{equation}\label{ieqs-in-Y}
	\sum_{j_1=1}^{d_1} \cdots \widehat{\sum_{j_{i}=1}^{d_{i}}}  \cdots \sum_{j_{n}=1}^{d_{n}}    Y^{(i)}_{j_1 \cdots \hat{j_i} \cdots j_n} (k,l) \   p_{j_1  \cdots j_n} \geq 0.
\end{equation}

For fixed $i \in [n], k,l \in [d_i]$, let $U^{(i)}_{k l} \in \R^{d_1 \cdots d_n}$ be the vector with entries
\[
	(U^{(i)}_{kl})_{j_1 \dots j_i \dots j_n} = 	\begin{cases}
		Y^{(i)}_{j_1 \dots \hat{j_i} \dots j_n}(k,l) & \text{if } j_i = k  \text{ and } k<l \\[10pt]
		-Y^{(i)}_{j_1 \dots \hat{j_i} \dots j_n}(k,l) & \text{if } j_i = k  \text{ and } k>l \\[10pt]
		0 & \text{otherwise}
	\end{cases}
\]
for each coordinate indexed by $j_1 \in [d_1], \dots, j_i \in [d_i], \dots, j_n \in [d_n]$.
Using the same indexing of coordinates for $p \in \R^{d_1 \cdot \ldots \cdot d_n}$, 
the inequalities \eqref{ieqs-in-Y} can be expressed as the inner product $\langle U^{(i)}_{kl} , p \rangle \geq 0$.

\begin{example}[$(2 \times 2)$-games]\label{ex:general-2x2-game}
	Consider a $2$-player game with $d_1 = d_2 = 2$. We fix the indexing $p = (p_{11}, p_{12}, p_{21}, p_{22})$. 
	Recall that the inequalities are
		\begin{align*}
(X^{(1)}_{11}-X^{(1)}_{21}) p_{11} + (X^{(1)}_{12}-X^{(1)}_{22}) p_{12} &= Y^{(1)}_1(1,2)\ p_{11} + Y^{(1)}_2(1,2)\ p_{12} \geq 0 \\
(X^{(1)}_{21}-X^{(1)}_{11}) p_{21} + (X^{(1)}_{22}-X^{(1)}_{12} ) p_{22} &= 	-Y^{(1)}_1(1,2)\ p_{21}  -Y^{(1)}_2(1,2)\ p_{22} \geq 0 \\
(X^{(2)}_{21}-X^{(2)}_{22}) p_{21} + (X^{(2)}_{11}-X^{(2)}_{12}) p_{11} 
			&= Y^{(2)}_2(1,2)\ p_{21} + Y^{(2)}_1(1,2)\ p_{11}
			\geq 0 \\
(X^{(2)}_{22}-X^{(2)}_{21}) p_{22} + (X^{(2)}_{12}-X^{(2)}_{11}) p_{12} 
			&= -Y^{(2)}_2(1,2)\ p_{22} - Y^{(2)}_1(1,2)\ p_{12}
			\geq 0. 
		\end{align*}
	The vectors $U^{(i)}_{kl}$ have entries in the $4$ unknowns $Y^{(1)}_1(1,2)\ , Y^{(1)}_2(1,2)\ , Y^{(2)}_1(1,2)\ , Y^{(2)}_2(1,2)$. More specifically,
	\begin{align*}
		U^{(1)}_{12} &= (Y^{(1)}_1(1,2)\ , Y^{(1)}_2(1,2)\ , 0 , 0) \\
		U^{(1)}_{21} &= -(0, 0, Y^{(1)}_1(1,2)\ , Y^{(1)}_2(1,2)) \\
		U^{(2)}_{12} &= (Y^{(2)}_1(1,2)\ , 0 , Y^{(2)}_2(1,2)\ , 0) \\
		U^{(2)}_{21} &= -(0, Y^{(2)}_1(1,2)\ , 0 , Y^{(2)}_2(1,2)).
	\end{align*}
The cone $C_G$ is defined by the inequalities $\langle U^{(i)}_{kl}, p \rangle \geq 0$ for $i \in [2]$, and $k, l \in [2], k \neq l$,  and the inequalities $\langle e_i, p  \rangle \geq 0$, where $e_{i}$ denote the standard basis vectors of $\R^{4}$. \\
\end{example}

Recall that the number of variables defined above is $ M = \sum_{i = 1}^n \binom{d_{i}}{2} \prod_{\substack{k \in [n] \\ k \neq i}} d_k$.
The ambient dimension of the correlated equilibrium polytope and cone is $D = \prod_{i=1}^n d_i$, and
the number of linear inequalities of the form \eqref{ieqs-in-Y} is $N = \sum_{i=1}^n d_i(d_i - 1)$. Let $U = U(Y) \in \R^{N \times D}$ be the matrix with rows $U^{(i)}_{kl}$ for $i \in [n], k,l \in [d_i], k\neq l$, and let $A(Y) \in \R^{(D+N)\times D}$ be the block matrix
\[
    A(Y) = \ma U \\ Id_{D} \trix.
\]
By \eqref{ieqs-in-Y}, the cone $C_G = C(Y)$ is given by
\[
    C(Y) = \{p \in \R^D \mid A(Y)  p \geq 0 \}.
\]

The matrix $A(Y)$ has full rank $D$, and so $C(Y)$ is a pointed cone. \\

For $(d_1 \times \dots \times d_n)$-games, where $d_i \geq 3$ for some $i \in [n]$, we have additional relations
\begin{equation}\label{linear-space}
    Y^{(i)}_{j_1\cdots\hat{j_i}\cdots j_n}(k,l) + Y^{(i)}_{j_1\cdots\hat{j_i}\cdots j_n}(l,t) = Y^{(i)}_{j_1\cdots\hat{j_i}\cdots j_n}(k,t)
\end{equation}
for $j_1 \in [d_1], \dots, j_{i-1} \in [d_{i-1}], \ k,l,t \in [d_i],\  j_{i+1} \in [d_{i+1}], \dots, j_n \in [d_n]$. A vector $Y \in \R^M$ corresponds to a certain game $G$ if and only if these relations hold. Geometrically, these relations define a linear space inside $\R^M$. We thus make the following definition.

\begin{definition}
The \emph{correlated equilibrium space} $\CES \subseteq \R^M$ of $(d_1 \times \dots \times d_n)$-games is the linear space defined by the equations \eqref{linear-space}, where $i \in [n]$ ranges over all players with at least $3$ strategies $d_i \geq 3$. If all players have at most $2$ strategies, then no such relation among the variables holds, and the correlated equilibrium space is the entire ambient space $\CES = \R^M$.
\end{definition}

\begin{example}[$\CES$ for $(2\times 3)$-games]
 In a $(2\times 3)$-game, there are six variables
 \begin{align*}
    Y^{(1)}_1(1,2) = X^{(1)}_{11} - X^{(1)}_{21}, \qquad
    Y^{(1)}_2(1,2) = X^{(1)}_{12} - X^{(1)}_{22}, \qquad
    Y^{(1)}_3(1,2) = X^{(1)}_{13} - X^{(1)}_{23}, \\
    Y^{(2)}_1(1,2) = X^{(2)}_{11} - X^{(2)}_{12}, \qquad
    Y^{(2)}_1(1,3) = X^{(2)}_{11} - X^{(2)}_{13}, \qquad
    Y^{(2)}_1(2,3) = X^{(2)}_{12} - X^{(2)}_{13}, \\
    Y^{(2)}_2(1,2) = X^{(2)}_{21} - X^{(2)}_{22}, \qquad
    Y^{(2)}_2(1,3) = X^{(2)}_{21} - X^{(2)}_{23}, \qquad
    Y^{(2)}_2(2,3) = X^{(2)}_{22} - X^{(2)}_{23}. 
 \end{align*}
 The relations among these variables are
  \begin{align*}
    Y^{(2)}_1(1,2) + Y^{(2)}_1(2,3) = Y^{(2)}_1(1,3), \qquad
    Y^{(2)}_2(1,2) + Y^{(2)}_2(2,3) = Y^{(2)}_2(1,3).
 \end{align*}
 These relations cut out the correlated equilibrium space $\CES$ for $(2 \times 3)$-games.
 For any game $(2\times 3)$-games $G$ the correlated equilibrium polytope is nonempty, a point $Y \in \R^6$ defines a correlated equilibrium cone if and only if it satisfies the above relations.
\end{example}

\begin{remark}
    In the following sections, we will classify correlated equilibrium polytopes and cones with respect to the variables $Y$ instead of the payoff tensors $X$. We note that this is not a significant restriction, as this is a linear change of coordinates and thus does not change the geometry of the objects described in what follows, provided one restricts to the correlated equilibrium space $\CES$.
\end{remark}

\section{The Region of Full-Dimensionality}\label{sec:full-dim}

In this section, we introduce the region of full-dimensionality for a type of game. For fixed $d_i \in~\mathbb N, i \in [n]$, this region classifies for which $(d_1 \times \dots \times d_n)$-games the polytope $P_G$ is of maximal dimension $D - 1 = d_1 \cdots d_n - 1$. 
The connections between full-dimensionality and elementary games are discussed in \cite[Section 3]{viossat_elementarygamesgames}. The game is elementary if and only if none of the incentive constraints is
vacuous and $P_G$ has full dimension. In general, it is not understood under which conditions $P_G$ has maximal dimension (i.e.\ $G$ is a \emph{full game}), though there are some partial results on forbidden dimensions \cite[Proposition 7]{viossat10_propertiesapplicationsdual}. Exploring full-dimensional correlated equilibrium polytopes provides valuable insights from a game theoretical perspective as in this case, Nash equilibria lie on the relative boundary of $P_G$ (\Cref{prop: nash equilibria on boundary}). Since computing Nash equilibria is PPAD-complete \cite{daskalakis2009complexity} and correlated equilibria are computationally less expensive, it is a meaningful pursuit to study the combinatorial properties of full-dimensional $P_G$. 

Let $Y \in \CES \subseteq \R^M$ be a vector of indeterminates, as described in \Cref{sec:correlated-cone}. We consider $A = A(Y)$ as a matrix with indeterminates, where each choice of $Y \in \R^M$ uniquely defines a correlated equilibrium cone
\[
    C(Y) = \{p \in \R^N \mid A(Y)  p \geq 0 \}.
\]

Recall that the correlated equilibrium polytope $P_G(Y)$ is of maximal dimension if and only if $C(Y)$ is full-dimensional. Thus, we are interested in the \emph{region of full-dimensionality}
\[
    \mathcal D = \{Y \in \CES \subseteq \R^M \mid \dim(C(Y))=D \}.
\]

\begin{definition}
    A \emph{semialgebraic set} is a subset of $\R^M$ defined by a boolean combination of finitely many polynomial inequalities.
\end{definition}

\begin{theorem}\label{th:full-dim-semialgebraic}
    The region $\mathcal  D$ of full-dimensionality is the semialgebraic set
    \[
    \pi_Y(\{(x,Y) \in \R^{D+M} \mid A(Y)x>0 \}) \cap \CES
    \]
    where $\pi_Y$ is the coordinate projection onto the last $M$ coordinates. 
\end{theorem}
\begin{proof}
    The cone $C(Y)$ is full-dimensional if and only if it has nonempty interior, i.e.\ if there exists some $p \in \R^D$ such that $A(Y)p >0$. Let $x = (x_{j_1 \cdots j_n} \mid i \in [n], j_i \in [d_i])$ be a vector of $D$ indeterminates. Consider the set
    \[
        \{(x,Y) \in \R^{D+M} \mid A(Y)x>0 \}.
    \]
    The expression $A(Y)x$ defines a $(D+N)$-dimensional vector, where each coordinate is a polynomial in variables $x$ and $Y$. Hence, the set $\{(x,Y) \in \R^{D+M} \mid A(Y)x>0 \}$ is defined by $D+N$ polynomial inequalities, and is thus a (basic) semialgebraic set. The region $\mathcal D$ of full-dimensionality is the intersection of the correlated equilibrium space $\CES$ with a coordinate projection of this set, which can be obtained by projecting away the $x$-coordinates. The Tarski–Seidenberg theorem implies that the coordinate projection is semialgebraic, and hence $\mathcal D$ is a semialgebraic set. 
\end{proof}

\begin{example}[$\mathcal D$ for $(2 \times 2)$-games]\label{ex:full-dim-2x2}
    Consider a $(2 \times 2)$-game. In this case, the ambient dimension of the correlated equilibrium polytope is $D = 4$, the number of incentive constraints is $N = 4$ (so the number of inequalities that define the polytope is $D+N = 8$) and the ambient dimension of $\mathcal D \subseteq \CES = \R^M$ is $M = 4$. The different combinatorial types in this case have been fully classified in \cite{calvoarmengol03_setcorrelatedequilibria}. Here, $\mathcal D \subseteq \R^4$ is the union of two open orthants:
    \begin{align*}
        a) \ &Y^{(1)}_{1}(1,2) > 0, \ Y^{(1)}_{2}(1,2) < 0, \  Y^{(2)}_{1}(1,2) > 0 , \ Y^{(2)}_{2}(1,2) < 0 \\
        b) \ &Y^{(1)}_{1}(1,2) < 0, \ Y^{(1)}_{2}(1,2) > 0 , \ Y^{(2)}_{1}(1,2) < 0 , \ Y^{(2)}_{2}(1,2) > 0
    \end{align*}
    The file \texttt{dimensions2x2.nb} in \cite{mathrepo} contains \texttt{Mathematica} code \cite{Mathematica} which also computes these inequalities.
\end{example}

\begin{example}[$\mathcal D$ for $(2 \times 3)$-games]\label{ex:full-dim-2x3}
    The coordinate projection of the set $\{(x,Y) \in \R^{D+M} \mid A(Y)x>0 \}$ onto $\R^M$ is the union of basic semialgebraic sets, where each piece is the intersection of an orthant with a binomial inequality. 
    One of the pieces is
    \begin{align*}
        &Y^{(1)}_{1}(1,2) > 0, \ Y^{(1)}_{2}(1,2) > 0,    \ Y^{(1)}_{3}(1,2) < 0,    \\
        &Y^{(2)}_{1}(1,2) < 0, \ Y^{(2)}_{1}(1,3) < 0, \ Y^{(2)}_{1}(2,3) > 0, \ Y^{(2)}_{2}(1,2) > 0,    \ Y^{(2)}_{2}(1,3) > 0,    \ Y^{(2)}_{2}(2,3) < 0, \\ \
        &Y^{(2)}_{2}(1,3) \ Y^{(2)}_{1}(2,3) < Y^{(2)}_{1}(1,3) \ Y^{(2)}_{2}(2,3).
    \end{align*}
    The region $\mathcal D$ of full-dimensionality for $(2 \times 3)$-games consists of the intersection of the above mentioned pieces with the correlated equilibrium space $\CES$.
    The \texttt{Mathematica} file \texttt{dimensions2x3.nb} in \cite{mathrepo} contains our code for computing all of the components of this semialgebraic set. We note that the orthants appearing in $\mathcal{D}$ seem highly structured. 
\end{example}

\begin{remark}\label{rmk:full-dim-symmetry}
The structured behavior of the regions of full-dimensionality that we see in $(2\times 3)$-games arises more generally for arbitrary $(2\times n)$-games. Exploiting the structure of the matrix $A(Y)$, 
the region $\mathcal{D}$ of full-dimensionality for $(2\times n)$-games \textbf{cannot} satisfy any of the following conditions:
\begin{itemize}
    \item $Y_k^{(1)} (1,2)> 0$, for all $k \in [n]$.
    \item $Y_k^{(1)} (1,2)< 0$, for all $k \in [n]$. 
    \item $Y_1^{(2)} (k,n)>0$ and $Y_2^{(2)} (k,n)>0$, for all $k \in [n]$
    \item $Y_1^{(2)} (k,n)<0$ and $Y_2^{(2)} (k,n)<0$, for all $k \in [n]$.\\
\end{itemize}
\end{remark}

While this approach could theoretically be used to obtain inequalities for larger games, this is extremely difficult in practice since the required algebraic methods do not scale well as the number of variables involved increases. For example, we were unable to carry out this computation for $(2\times 2 \times 2)$-games since we must compute the coordinate projection of a semialgebraic set that lives in $D+M = 20$-dimensional space.

\section{Combinatorial Types of Correlated Equilibrium Polytopes}\label{sec:comb-types}

In this section, we consider how to systematically classify combinatorial types of polytopes arising as a correlated equilibrium polytope. The {\emph{face lattice}} of a polytope $P$ is the poset of all the faces of $P$, partially ordered by inclusion. Two polytopes have the same {\emph{combinatorial type}} if there exists an isomorphism between their face lattices. \\
First, we present a systematic approach for classifying the possible combinatorial types for arbitrary games by describing the stratification by oriented matroids. Even for small examples, the explicit computation of the oriented matroid strata is beyond current scope, but we introduce algebraic methods for understanding the oriented matroid strata via their algebraic boundaries. 
We use this technique to completely classify the combinatorial types of $P_G$ for $(2\times 3)$-games (\Cref{thm:2x3typesFromStrata}). We then show that for all $(2\times n)$-games the irreducible components of the algebraic boundary of the oriented matroid stratification are coordinate hyperplanes, $(2\times 2)$-minors and $(4 \times 4)$-minors of the matrix $A(Y)$  (\Cref{th:binomials}).

\begin{example}[Hawk-Dove game]\label{ex: correlated polytope}

This game models a scenario of a competition for a shared resource. Both players can choose between conflict (hawk) or conciliation (dove). This game is a generalization of the Traffic Lights example (\Cref{ex:traffic-lights-1,ex:traffic-lights-2,ex:traffic-lights-3}). The inequalities for general $(2 \times 2)$-games are given in \Cref{ex:general-2x2-game}. In the Hawk-Dove game, each of the two players has two strategies $s^{(i)}_1$=``hawk'' or $s^{(i)}_2$=``dove''. 
 \begin{table}[h]
    \setlength{\extrarowheight}{2pt}
    \begin{tabular}{cc|c|c|}
      & \multicolumn{1}{c}{} & \multicolumn{2}{c}{Player $2$}\\
      & \multicolumn{1}{c}{} & \multicolumn{1}{c}{Hawk}  & \multicolumn{1}{c}{Dove} \\\cline{3-4}
      \multirow{2}*{Player $1$}  & Hawk & $\left(\frac{V-C}{2},\frac{V-C}{2}\right)$ & $(V,0)$ \\\cline{3-4}
      & Dove & $(0,V)$ & $\left(\frac{V}{2},\frac{V}{2}\right)$ \\\cline{3-4}
    \end{tabular} \hspace*{1cm}
  \end{table}
  
In this bimatrix game, $V$ represents the value of the resource and $C$ represents the cost of the escalated fight. It is mostly assumed that $C > V>0$. 
The correlated equilibria polytope $P_G$ is full-dimensional with $5$ vertices and $6$ facets. 
In the case $V\geq C>0$, the game becomes an example for {\emph{Prisoner's Dilemma}}, in which case $P_G$ is a single point. 
\end{example}

As seen in the previous example, for fixed $d_1, \dots, d_n \in \N$, different choices of the payoffs for the players in a $(d_1 \times \dots \times d_n)$-game can result in different combinatorial types of correlated equilibria. 
We would thus like to classify the regions of the correlated equilibrium space $\mathcal{S} \subseteq \R^M$ such that 
\[
    \{Y \in \CES \mid C(Y) \text{ has a fixed combinatorial type} \}.
\]

We now explain how the combinatorial type of $C(Y)$ is completely determined by the underlying oriented matroid defined by the matrix $A(Y)$.
The combinatorial type of $C(Y)$ is the incidence structure of rays and facets of $C(Y)$. Equivalently, we can classify the incidence structure of facets and rays of the dual cone $C(Y)^\ast$.
The \emph{dual} of the cone $C(Y)$ is defined as 
\[
    C(Y)^* = \{q \in \R^{d_1 \cdots d_n} \mid \langle p, q \rangle \geq 0, \  \forall p \in C(Y) \} .
\]
By definition, the (inner) facet normals of $C(Y)$ are generators of extremal rays of $C(Y)^*$ and vice versa. 
Let $h \in [D+N]$ and $A_h(Y)$ be a row of $A(Y)$. Seen as a linear functional, this row uniquely selects a face
\[
    F_h = \{p \in C(Y) \mid \langle A_h(Y) , p \rangle = 0 \}.
\]
Note that this is not necessarily a facet of $C(Y)$, but all facets of $C(Y)$ arise in this way. 
For the dual cone $C(Y)^\ast$ this implies that $r_h = A_h(Y) \R_{\geq 0}$ is an extremal ray of $C(Y)^*$ if and only if $F_h$ is a facet of $C$.
If $C(Y)$ is not full-dimensional, then there is some $h \in [D+N]$ such that $F_h = C(Y)$. The set of all such vectors span the \emph{lineality space}
\[
    \mathcal L= C(Y)^* \cap (-C(Y)^*) = \text{span}(\{A_h \mid F_h = C(Y)\})
\]
of $C^*(Y)$.
In this case, extremal rays of $C(Y)^*$ are to be considered in $C(Y)^* / \mathcal L$.\\

We want to understand the incidence structure of extremal rays and facets of $C(Y)$. Each such ray of $C(Y)$ is contained in at least $D-1$ facets. Thus, we seek to understand which subsets of $D-1$ faces $F_h$ of $C(Y)$ intersect in a single point. Equivalently, we want to understand which subsets of $D-1$ rays $r_h$ of $C^*(Y)$ are contained in a common face. Let $H \subseteq [D+N]$ such that $|H|=D-1$, and denote by $A_H(Y)$ the submatrix of $A(Y)$ with rows indexed by $H$. If $\{r_h \mid h \in H\}$ lie on a common face, then these rays all lie on the hyperplane given by the rowspan of $A_H(Y)$. Additionally, let $h' \in [D+N]\setminus H$. Then all $r_{h'}$ lie on the same side of the hyperplane. This implies that the sign of $\det(A_{H \cup h'})$ is uniquely determined for all $h' \in [D+N]\setminus H$. 
The collection $\mathcal{R}$ of all regions in which the maximal minors of $A(Y)$ satisfy a certain sign pattern is known as the \emph{oriented matroid stratification} of $\R^M$. Each cell gives rise to a fixed sign pattern of $A(Y)$, i.e.\ the underlying oriented matroid. Restricting the oriented matroid strata to the correlated equilibrium space $\CES$ yields a subdivision of $\CES$ in which distinct combinatorial types lie in distinct regions.

\begin{definition}
    Let $S \subseteq \R^M$ be a semialgebraic set. 
    The \emph{algebraic boundary} $\partial_a S$ is the Zariski closure of the topological (euclidean) boundary 
    $\partial S$, i.e.\ the smallest algebraic variety containing $\partial S$ (over $\mathbb C$).
\end{definition}

\begin{construction}[Algebraic Boundary]
    For every $H \subseteq [D+N], |H|=D$ the minor $\det(A_H)$ is a polynomial in variables $Y$ of degree at most $D$, and $\text{sgn}(\det(A_H)) \in \{-1, 0, 1\}$ is a polynomial inequality.
   Let $s = (s_H \mid H \subseteq \binom{[D+N]}{D} ), s_H \in \{1,-1\}$ be a sign vector. Each maximal open region in the oriented matroid stratification is of the form
		\[
			\mathcal R^\circ_s =  \left\{Y \in \R^M \mid \text{sgn}(\det(A_H(Y))) =s_H \text{ for all } H \in \binom{[D+N]}{D} \right\}.
		\]
		Let $\mathcal R_s$ denote the Euclidean closure of such an open maximal region, which is a closed basic semialgebraic set.
		We define the \emph{algebraic boundary of the oriented matroid stratification} $\partial_a \mathcal R$ as the union of the algebraic boundaries of all such closed regions, i.e.
  	\[
		\partial_a \mathcal{R} = \bigcup_{s} \partial_a \mathcal R_s =   \bigcup_{H \in \binom{[D+N]}{D}} \mathcal V(\det(A_H(Y))),
		\]
  where the first union ranges over all sign vectors $s$ such that $\mathcal R^\circ_s$ is maximal, and
		$$\mathcal V(\det(A_H(Y))) = \{Y \in R^M \mid \det(A_H(Y)) = 0  \}.$$
		In the second union
   we only consider $H \in \binom{[D+N]}{D}$ such that $A_H(Y)$ contains at least one variable. Recall that $A(Y) = \sma U \\  Id_N \strix$ where the matrix $U$ has no zero rows and all nonzero entries are variables. Thus, 
   the only minor of $A(Y)$ that does not contain a variable is the determinant of $Id_N$, so the number of such minors is $\binom{D+N}{D}-1$.
   We note that the minors may not be irreducible polynomials, so they do not necessarily define the irreducible components of the variety $\partial_a \mathcal{R}$. 
   In total, this defines $\binom{D+N}{D}-1$ hypersurfaces, whose defining polynomials are of degree at most $D$.
\end{construction}

Applying \cite{basu2019betti} and \cite[Theorem 4]{basu2003different} to this setup yields a general bound on the number of strata on the oriented matroid stratification.
\begin{proposition}\label{prop:terrible-bound}
    Let $\delta$ be the maximum degree of all $\beta \leq \binom{D+N}{D}-1$ defining polynomials. 
    Then $\delta \leq D$ and the number of strata in the oriented matroid stratification is at most
    \[
        \delta (2 \delta-1)^{M-1}(1 + 3 \beta  ).
    \]
\end{proposition}

The following three examples illustrate this construction for small games.

\begin{example}[$\partial_a \mathcal{R}$ for $(2\times 2)$-games]\label{ex:2x2strata}
In a $(2\times 2)$-game we have $D = 4$, $N = 4$ and $M = 4$. The number of irreducible components of $\partial_a \mathcal{R}$ is $\beta = 4$, and these are the $4$ coordinate hyperplanes. Indeed, the classification in \cite{calvoarmengol03_setcorrelatedequilibria} shows that in each open orthant, the combinatorial type is fixed, and in the two orthants described in \Cref{ex:full-dim-2x2} there is a unique combinatorial type of maximal dimension. The file \texttt{orientedMatroidStrata2x2.m2} in \cite{mathrepo} contains \texttt{Macaulay2} code \cite{M2} which explicitly computes these irreducible components. 
\end{example}

\begin{example}[$\partial_a \mathcal{R}$ for $(2\times 3)$-games]
\label{ex:2x3strata}
    In a $(2 \times 3)$-game, we have $D = 6, N = 8$ and $M = 9$.
    The number of minors of $A_H(Y)$ that contain at least one variable is $\binom{12}{6}-1 = 934$, but the number of irreducible components is only $\beta = 12$. All of these polynomials are homogeneous, and the maximum degree is $\delta = 2$. 
    In fact, the irreducible components are the $9$ coordinate hyperplanes, together with the $3$ binomials
    \begin{align*}
        Y^{(2)}_{2}(1, 2) \ Y^{(2)}_{1}(1, 3) - Y^{(2)}_{1}(1, 2) \ Y^{(2)}_{2}(1, 3) \\
        Y^{(2)}_{2}(1, 2) \ Y^{(2)}_{1}(2, 3) - Y^{(2)}_{1}(1, 2) \ Y^{(2)}_{2}(2, 3) \\
        Y^{(2)}_{2}(1, 3) \ Y^{(2)}_{1}(2, 3) - Y^{(2)}_{1}(1, 3)  \ Y^{(2)}_{2}(2, 3).
    \end{align*}
    
    Thus, \Cref{prop:terrible-bound} implies that the number of strata in the oriented matroid stratification is at most
    $
        2 \cdot 3^{8}(1 + 3 \cdot 12) = 485514.
    $
    However, as we will show in \Cref{thm:2x3typesFromStrata}, it turns out that there are precisely $3$ distinct combinatorial types.
    Note that the three binomials above are precisely the binomials intersecting the orthants in \Cref{ex:full-dim-2x3}. The file \texttt{orientedMatroidStrata2x3.m2} in \cite{mathrepo} contains \texttt{Macaulay2} code which explicitly computes these polynomials. 
\end{example}

\begin{example}[$\partial_a \mathcal{R}$ for $(2\times 2 \times 2)$-games]
    In a $(2 \times 2 \times 2)$-game, we have $D = 8, N = 6$ and $M = 12$.
    The number of minors of $A_H(Y)$ that contain at least one variable is $\binom{14}{8}-1 = 3003$, but the number of irreducible components is only $194$. All of these polynomials are homogeneous, and the maximum degree is $6$. The file \texttt{orientedMatroidStrata2x2x2.m2} in \cite{mathrepo} contains \texttt{Macaulay2} code which explicitly computes these polynomials. 
    \Cref{prop:terrible-bound} implies that the number of strata of the oriented matroid stratification is at most
    \[
        6 \cdot 11^{11} (1 + 3 \cdot 194) = 998 \ 020 \ 223 \ 797 \ 278 > 10^{14}.
    \]
\end{example}

The previous examples illustrate that the algebraic boundary of the oriented matroid stratification is quite nice for $(2\times n)$-games but becomes significantly more complicated even for $(2\times 2 \times 2)$-games. The following theorem shows that the nice structure we see for $(2\times 3)$-games holds for all $(2 \times n)$-games. 

\begin{theorem}\label{th:binomials}
    Consider a $(2 \times n)$-game, i.e.\ a $2$-player game with $d_1 = 2, d_2 = n \in \N$. Then the irreducible components of $\partial_a \mathcal{R}$ are given by
    \begin{itemize}
        \item the coordinate hyperplanes,
        \item hypersurfaces defined by quadratic binomials that are given by certain $(2 \times 2)$-minors
        and $(4\times 4)$-minors
        of the matrix $A(Y)$.
    \end{itemize}
\end{theorem}

\begin{proof}
We prove by induction on $n$. For $n = 2$ after reordering the rows and columns, the matrix $A(Y) = A^2(Y)$ has the following representation:

\[
\left(
\begin{tabular}{|rr|rr|} 
\hline
$Y^{(1)}_{1}(1,2)$      &                     & $Y^{(1)}_{2}(1,2)$ &                       \\
                        & $-Y^{(1)}_{1}(1,2)$ &                  & $-Y^{(1)}_{2}(1,2)$     \\ 
\hline
$Y^{(2)}_{1}(1,2)$      &  $Y^{(2)}_{2}(1,2)$ &                  & \multicolumn{1}{l}{}    \\
\cline{1-4}
\multicolumn{1}{l}{}    &                   & $Y^{(2)}_{1}(2,1)$ & $Y^{(2)}_{1}(2,1)$      \\
\hline 
\multicolumn{4}{|c|}{\begin{tabular}[c]{@{}c@{}}\\$Id_4$\\ \\  \end{tabular}}      \\
\hline
\end{tabular}
\right)
\]

For $n=3$ the columns and rows can be arranged to $A(Y) = A^3(Y)$:

\[\left(
\begin{tabular}{|rr|rr|rr|} 
\hline
$Y^{(1)}_{1}(1,2)$   &                      & $Y^{(1)}_{2}(1,2)$ &                      & $Y^{(1)}_{3}(1,2)$ &                       \\
                     & $- Y^{(1)}_{1}(1,2)$ &                    & $- Y^{(1)}_{2}(1,2)$ &                    & $-Y^{(1)}_{3}(1,2)$    \\ 
\hline
$Y^{(2)}_{1}(1,2)$   & $Y^{(2)}_{2}(1,2)$   &                    & \multicolumn{1}{l}{} &                    & \multicolumn{1}{l}{}  \\
$Y^{(2)}_{1}(1,3)$   & $Y^{(2)}_{2}(1,3)$   &                    & \multicolumn{1}{l}{} &                    & \multicolumn{1}{l}{}  \\ 
\cline{1-4}
\multicolumn{1}{l}{} &                      & $Y^{(2)}_{1}(2,1)$ & $Y^{(2)}_{2}(2,1)$   &                    & \multicolumn{1}{l}{}  \\
\multicolumn{1}{l}{} &                      & $Y^{(2)}_{1}(2,3)$ & $Y^{(2)}_{2}(2,3)$   &                    & \multicolumn{1}{l}{}  \\ 
\cline{3-6}
\multicolumn{1}{l}{} & \multicolumn{1}{l}{} &                    &                      & $Y^{(2)}_{1}(3,1)$ & $Y^{(2)}_{2}(3,1)$    \\
\multicolumn{1}{l}{} & \multicolumn{1}{l}{} &                    &                      & $Y^{(2)}_{1}(3,2)$ & $Y^{(2)}_{2}(3,2)$    \\ 
\hline
\multicolumn{6}{|c|}{\begin{tabular}[c]{@{}c@{}}\\$Id_6$\\ \\  \end{tabular}}                                                               \\ 
\hline
\end{tabular}.
\right)\]

In both cases, the statement holds by \Cref{ex:2x2strata,ex:2x3strata}. First, we describe the general block matrix structure of $A(Y) = A^n(Y)$ for fixed $n \in \N$. Recall that $A^n(Y)$ is a $((D+N) \times D)$-matrix, where $D = 2n$ and $N = 2+n(n-1)$.
As in the case $n=2,3$, we can arrange the rows and columns of $A^n(Y)$ such that the first two rows consist of $n$ blocks of size $(2\times 2)$, which are of the form $\left[ \begin{smallmatrix}
     Y^{(1)}_{k}(1,2) &  \\  & -Y^{(1)}_{k}(1,2) \end{smallmatrix}
\right]$ for $k \in [n]$.
The following $n(n-1)$ rows form a block diagonal matrix, with $n$ blocks of size $((n-1) \times 2)$ of the form
\[
  b_k =  \left[ \begin{smallmatrix}
        Y^{(2)}_1(k,1) & Y^{(2)}_2(k,1) \\
        \vdots & \vdots \\[5pt]
        Y^{(2)}_1(k,n) & Y^{(2)}_2(k,n)
    \end{smallmatrix} \right]
\]
for each $k \in [n]$, where the row $[Y^{(2)}_1(k,k)  \  Y^{(2)}_1(k,k)]$ is omitted. Finally, the last $2n$ rows consist of the identity matrix $Id_{2n}$. 

We want to show that the maximal minors of the matrix $A^{n}(Y)$ are either monomials, binomials, or zero. Note that every maximal minor of $A^{n}(Y)$ that involves a row from the identity matrix corresponds to a smaller minor of the remaining matrix. We thus consider all $(m \times m)$-minors of the submatrix of $A^{n}(Y)$ that consists of the first $2+(n-1)n$ rows, for $m \leq 2n$. Let $M$ be such a $(m \times m)$-submatrix.   

If $M$ contains a row or column which has at most one entry, then computing $\det(M)$ by expanding this row or column implies the statement by induction on $m$. Thus, suppose that $M$ has at least $2$ nonzero entries in each row and column. If $M$ contains at least $3$ rows from one block of the block diagonal matrix, then these $3$ rows are linearly dependent (since each block only consists of $2$ columns), in which case $\det(M) = 0$.
On the other hand, if $M$ contains at most $2$ rows from each block, and neither the first nor the second row, then $M$ either contains a zero row, a zero column, or is a block-diagonal matrix with blocks of sizes $(2 \times 2)$ and $(1 \times 1)$, so $\det(M)$ is a product of binomials and monomials.

We are left with the case in which each row and column contains at least $2$ non-zero entries, and $M$ contains the first or second row. We first argue that $M$ must contain both the first and second row to meet these requirements.
Let $M_{*j}$ be a column of $M$. Since $M_{*j}$ contains at least two nonzero entries, at least one entry $M_{ij}$ is neither in the first nor second row of $A^n(Y)$. Thus, there exists a block $b_k, k \in [n]$ containing $M_{ij}$. Recall that the columns of $A^n(Y)$ containing the block $b_k$ are the columns with index $2k-1$ and $2k$. Since each row of $M$ contains at least two nonzero entries, $M$ must contain both of these columns. Following this argument for all columns of $M$ implies that $m$ is even. 

If only one of the first two rows is contained in $M$, then by the Pigeonhole principle there is at least one block such that $M$ has an odd number of rows from this block. As we have already covered the case in which $M$ has at least three rows from one block, we can assume there is a block $b_k$ of which $M$ contains precisely one row $M_{i*}$ with precisely two nonzero entries $M_{ij}$ and $M_{i,j+1}$. However, then either the column $M_{*j}$ or $M_{*j+1}$ contain only a single nonzero entry, so $M$ does not meet our assumptions. Thus, we can assume that $M$ contains both the first and second row of $A^n(Y)$, and $m-2$ rows from blocks of the block diagonal. 

We have seen that the assumption of having at least $2$ nonzero entries in each row implies that $M$ contains either $0$ or $2$ columns from each block $b_k$. The matrix $M$ thus contains $m-2$ rows from $\frac{m}{2}$ blocks in total, with at least $1$ and at most $2$ rows per block. This implies that there are precisely $2$ blocks which contribute a single row, while $M$ contains $2$ rows of the remaining $\frac{m}{2} -2$ blocks.

To summarize, the matrix $M$ is of the shape
\[ \left( 
\begin{tabular}{lllllllll}
\cline{1-6} \cline{8-9}
\multicolumn{1}{|l}{$\ast$} & \multicolumn{1}{l|}{}  & $\ast$ & \multicolumn{1}{l|}{}  & $\ast$ & \multicolumn{1}{l|}{}  & \multicolumn{1}{l|}{\multirow{2}{*}{\dots}} & $\ast$ & \multicolumn{1}{l|}{}  \\
\multicolumn{1}{|l}{}  & \multicolumn{1}{l|}{$\ast$} &   & \multicolumn{1}{l|}{$\ast$} &   & \multicolumn{1}{l|}{$\ast$} & \multicolumn{1}{l|}{}                     &   & \multicolumn{1}{l|}{$\ast$} \\ \cline{1-6} \cline{8-9} 
\multicolumn{1}{|l}{$\ast$} & \multicolumn{1}{l|}{$\ast$} &   &                        &   &                        &                                           &   &                        \\ \cline{1-4}
                       & \multicolumn{1}{l|}{}  & $\ast$ & \multicolumn{1}{l|}{$\ast$} &   &                        &                                           &   &                        \\ \cline{3-6}
                       &                        &   & \multicolumn{1}{l|}{}  & $\ast$ & \multicolumn{1}{l|}{$\ast$} &                                           &   &                        \\
                       &                        &   & \multicolumn{1}{l|}{}  & $\ast$ & \multicolumn{1}{l|}{$\ast$} &                                           &   &                        \\ \cline{5-6}
                       &                        &   &                        &   &                        & \rotatebox{-45}{\dots}                                       &   &                        \\ \cline{8-9} 
                       &                        &   &                        &   &                        & \multicolumn{1}{l|}{}                     & $\ast$ & \multicolumn{1}{l|}{$\ast$} \\
                       &                        &   &                        &   &                        & \multicolumn{1}{l|}{}                     & $\ast$ & \multicolumn{1}{l|}{$\ast$} \\ \cline{8-9} 
\end{tabular}
\right),
\] 
\noindent where $\ast$ denotes a non-zero entry, i.e. below the first two rows $M$ is a block matrix with $2$ blocks of size $(1 \times 2)$ and $\frac{m}{2} -2$ blocks of size $(2\times 2)$.
We now proceed by induction on the (even) size $m$ of the matrix $M$. If $m=4$ then $M$ does not contain any $(2 \times 2)$-block and it can be computed that $\det(M)$ is the difference of two monomials of degree $3$. For general even $m$, note that we can consider $M$ as a block matrix of the form $M = \sma B & C \\ 0 & D \strix$, where $B$ is a matrix of size $(m-2)\times(m-2)$ and $D$ a $(2 \times 2)$-block. Recall that for such upper triangular block matrices holds $\det(M) = \det(B) \det(D)$. By induction, $\det(B)$ is a product of monomials and binomials, and $\det(D)$ is a binomial since $D$ is a $(2 \times 2)$-block. 
Throughout the above arguments, the binomials in the product always arose as determinants of $(2 \times 2)$-submatrices of a block or from $(4 \times 4)$-submatrices involving two columns of even index, two columns of odd index, the first two rows and two rows from distinct blocks.
Finally, the statement follows.
\end{proof}

We now show how the oriented matroid stratification of $(2 \times 3)$-games can be used to completely determine the possible combinatorial types of the polytope for payoffs $Y$ which are generic with respect to the algebraic boundary described in \Cref{th:binomials}. In other words, none of the polynomials described in \Cref{th:binomials} vanish on $Y$. 

\begin{theorem}
\label{thm:2x3typesFromStrata}
Let $G$ be a $(2 \times 3)$-game with generic payoffs $Y \in \CES \subseteq \R^9$ and let $P_G = P(Y)$ be its associated correlated equilibrium polytope. Then one of the following holds
\begin{itemize}
    \item $P_G$ is a point,
    \item $P_G$ is of maximal dimensional and of a unique combinatorial type,
    \item There exists a $(2 \times 2)$-game $G'$ such that $P_{G'}$ is full-dimensional and combinatorially equivalent to $P_G$. 
\end{itemize}
\end{theorem}

\begin{figure}[t]
    \centering
    \includegraphics[scale=0.6]{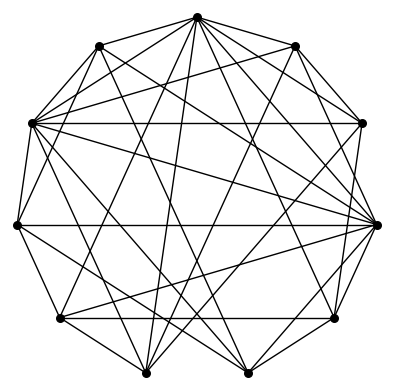}
    \caption{The graph of the combinatorially unique $5$-dimensional polytope that arises as the correlated equilibrium polytope of a $(2\times 3)$-game, as described in \Cref{thm:2x3typesFromStrata}.}
    \label{fig:2x3-graph}
\end{figure}

\begin{proof}

First recall that the combinatorial type is fully determined by the sign patterns of the maximal minors of $A(Y)$, i.e.\ the oriented matroid of $A(Y)$. 
For $(2 \times 3)$-games the matrix $A(Y)$ has 1206 nonzero maximal minors for generic $Y$, since $1797$ of the $\binom{14}{6} = 3003$ maximal minors are identically zero. 
The signs of these 1206 maximal minors are completely determined by the signs of the irreducible components of $\det(A(Y)_B)$ for each $B \in \binom{[14]}{6}$. As in \Cref{th:binomials} and \Cref{ex:2x3strata}, there are 12 such irreducible components $f_1, \ldots, f_{12}$, which are given by the 9 coordinate hyperplanes and the 3 binomials listed in the example. 
This means that once we fix a sign pattern $s \in \{1, -1\}^{12}$ of the $f_i$ the signs of all maximal minors of $A(Y)$ are uniquely determined. Thus, to compute all possible combinatorial types, we compute the combinatorial type of the polytope once for each possible sign pattern $s$ of the $f_i(Y)$.

We do this computationally, using the software \texttt{Mathematica 13.0} and \texttt{SageMath 9.6} \cite{Mathematica, sagemath}. We first use \texttt{Mathematica} to find a payoff $Y \in \CES$ such that $s_i f_i(Y) > 0$ for all $i=1,\dots,12$. We then compute the corresponding combinatorial type of the polytope in \texttt{SageMath}. These computations can be found in the files $\texttt{combinatorialTypes2x3.nb}$ and $\texttt{combinatorialTypes2x3.ipynb}$ on MathRepo \cite{mathrepo} respectively. 

As a result, we obtain $3$ different possible combinatorial types, which are a single point, the unique combinatorial type of full-dimensions of $(2\times 2)$-games (a bipyramid over a triangle), and a new unique full-dimensional combinatorial type. 
This polytope has $f$-vector $ (1, 11, 32, 40, 25, 8, 1)$, and the graph of this $5$-dimensional polytope is depicted in \Cref{fig:2x3-graph}. 
A full description of this polytope can be found in $\texttt{combinatorialTypes2x3.ipynb}$ on our MathRepo page \cite{mathrepo}.
\end{proof}

\begin{example}[Combinatorial types of $(2 \times 3)$-games]\label{ex:2x3-full-comb-type}
    By \Cref{thm:2x3typesFromStrata}, for generic $(2 \times 3)$-games there is a unique combinatorial type of full-dimension. This is a $5$-dimensional polytope with $f$-vector $(1, 11, 32, 40, 25, 8, 1)$ and its graph is depicted in \Cref{fig:2x3-graph}.
\end{example}

\Cref{th:binomials} shows, that in $(2 \times 3)$-games all correlated equilibrium polytopes that are not of maximal dimension appear as the maximal polytope of a smaller game. This gives rise to the following conjecture.

\begin{conjecture}
\label{conj:2xntypes}
    Let $G$ be a $(2\times n)$-game with generic payoff matrices and let $P_G$ be its correlated equilibrium polytope. If $P_G$ is not of maximal dimension, then there exists a $(2\times k)$-game $G'$ where $k < n$ such that $P_{G'}$ is has maximal dimension and $P_G$ and $P_{G'}$ are combinatorially equivalent. 
\end{conjecture}

\begin{table}
    \centering
\begin{tabular}{ |p{2cm}||p{1cm}|p{1cm}|p{1cm}|p{1cm}|p{1cm}|}
\hline
\multicolumn{6}{|c|}{Unique Combinatorial Types by Dimension} \\ 
\hline
Dimension & 0 & 3 & 5 & 7 & 9\\
\hline
\hline
$(2 \times 2)$ & 1 & 1 & 0 & 0 & 0 \\ 
$(2 \times 3)$ & 1 & 1 & 1 & 0 & 0 \\
$(2 \times 4)$  & 1 & 1 & 1 & 3 & 0 \\
$(2 \times 5)$ & 1 & 1 & 1 & 3 & 4  \\
\hline
\end{tabular}
\caption{The number of unique combinatorial types of $P_G$ of each dimension for a $(2 \times n)$-game in a random sampling of size $100 \ 000$.}
\label{table:typesOf2xn}
\end{table}

A relevant study to this conjecture is the dual reduction process of finite games. An iterative dual reduction reduces a finite game $G$ to a lower-dimensional elementary game $G'$, for which $P_{G'}$ is full-dimensional, by deleting certain pure strategies or merging several pure strategies into a single one. Any correlated equilibrium of the reduced game $G'$ is a correlated equilibrium of the original game $G$, however, the opposite is not always true \cite[Section 5]{myerson-dualreduction}. 
\Cref{conj:2xntypes} is supported by our computations thus far. To test this conjecture, we sampled $100 \ 000$ random payoff matrices $X$ for $(2\times n)$-games for $n = 4,5$. The results are summarized in \Cref{table:typesOf2xn}, which shows the number of unique combinatorial types of a given dimension that we found for each $(2\times n)$-game. In all of our computations, \Cref{conj:2xntypes} holds. The supporting code can be found on \cite{mathrepo}. 

In contrast to the $(2\times n)$-case, $(2\times 2 \times 2)$-games exhibit a much wider variety of distinct combinatorial types. In a sample of $100 \ 000$ random payoff matrices for $(2 \times 2 \times 2)$-games, we found $14 \ 949$ distinct combinatorial types which are of maximal dimension. Amongst these $7$-dimensional polytopes, the number of vertices can range from $8$ to $119$, the number of facets from $8$ to $14$, and the number of total faces from $256$ to $2338$. 
Examples of $f$-vectors achieving these bounds are
\begin{align*}
    f_{P_{G_1}} &= (1, 8, 28, 56, 70, 56, 28, 8, 1)\\
    f_{P_{G_2}} &= (1, 119, 458, 728, 616, 302, 87, 14, 1)
    , \\
    f_{P_{G_3}} &= (1, 119, 460, 733, 620, 303, 87, 14, 1)
    .
\end{align*}

\vspace{1em}

\bibliographystyle{alpha}
\bibliography{references}

\newpage

\vspace*{\fill}

\subsection*{Affiliations} 
$ $
\vspace*{0.2cm}

\noindent \textsc{Marie-Charlotte Brandenburg} \\
\textsc{KTH Royal Institute of Technology \\
Lindstedtsvägen 25, 114 28 Stockholm, Sweden} \\
 \url{mcbra@kth.se} \\

\noindent \textsc{Benjamin Hollering} \\
\textsc{ Technische Universit\"at M\"unchen \\
Boltzmannstr. 3, 85748 Garching b. München, Germany} \\
\url{benhollering@gmail.com} \\

 \noindent \textsc{Irem Portakal} \\
\textsc{ Max Planck Institute for Mathematics in the Sciences \\
Inselstr. 22, 04103 Leipzig, Germany} \\
\url{mail@irem-portakal.de}
\end{document}